\documentclass{article}
\usepackage[british]{babel}
\usepackage{amsmath,amssymb,graphicx,enumerate,xcolor,latexsym}
\usepackage[parfill]{parskip}

\newcommand{\tmaffiliation}[1]{\\ #1}
\newcommand{\tmcolor}[2]{{\color{#1}{#2}}}
\newcommand{\tmemail}[1]{\\ \textit{Email:} \texttt{#1}}
\newcommand{\tmop}[1]{\ensuremath{\operatorname{#1}}}
\newcommand{\tmstrong}[1]{\textbf{#1}}
\newcommand{\tmtextit}[1]{\text{{\itshape{#1}}}}
\newcommand{\tmtextsc}[1]{\text{{\scshape{#1}}}}
\newcommand{\tmtextup}[1]{\text{{\upshape{#1}}}}
\newenvironment{enumeratenumeric}{\begin{enumerate}[1.] }{\end{enumerate}}
\newenvironment{enumerateroman}{\begin{enumerate}[i.] }{\end{enumerate}}
\newenvironment{itemizedot}{\begin{itemize} }{\end{itemize}}
\newenvironment{proof}{\noindent\textbf{Proof\ }}{\hspace*{\fill}$\Box$\medskip}
\newtheorem{theorem}{Theorem}

\begin{document}

\title{A visualisation for conveying the dynamics of iterative eigenvalue
algorithms over PSD matrices}

\author{
  Ran Gutin
  \tmaffiliation{Department of Computer Science, Imperial College London}
  \tmemail{jkabrg@gmail.com, rg120@ic.ac.uk}
}

\maketitle

\begin{abstract}
  We propose a new way of visualising the dynamics of iterative eigenvalue
  algorithms such as the QR algorithm, over the important special case of PSD
  (positive semi-definite) matrices. Many subtle and important properties of
  such algorithms are easily found this way. We believe that this may have
  pedagogical value to both students and researchers of numerical linear
  algebra. The fixed points of iterative algorithms are obtained visually, and
  their stability is analysed intuitively. It becomes clear that what it means
  for an iterative eigenvalue algorithm to ``converge quickly'' is an
  ambiguous question, depending on whether eigenvalues or eigenvectors are
  being sought. The presentation is likely a novel one, and using it, a
  theorem about the dynamics of general iterative eigenvalue algorithms is
  proved. There is an accompanying video series, currently hosted on Youtube,
  that has certain advantages in terms of fully exploiting the interactivity
  of the visualisation.
\end{abstract}

\section{Simple iterative eigenvalue algorithms}

The (naive) QR algorithm {\cite{francis1961qr,wilkinson1968global,GoluVanl96}}
evaluated on matrix $M$ begins with setting $M_0 = M$ and then repeating the
following two steps until convergence:
\begin{enumeratenumeric}
  \item Find the QR decomposition $M_n = Q_n R_n$.
  
  \item Set $\text{} M_{n + 1} = R_n Q_n$.
\end{enumeratenumeric}
The thing to note is that $M_{n + 1} = Q_n^T M_n Q_n$. We only concern
ourselves with PSD matrices\footnote{ A PSD (positive semi-definite) matrix is
a symmetric $\mathbb{R}$-matrix whose eigenvalues are all non-negative.
Equivalently, it is a matrix $M$ for which $v^T M v \geq 0$ for all vectors
$v$.}, so the fixed points of the iteration above are all diagonal matrices
with non-negative real entries.

We've also investigated a variant of the LR algorithm evaluated on a PSD
matrix $M$ that begins with setting $M_0 = M$ and then repeating:
\begin{enumeratenumeric}
  \item Find the Cholesky decomposition $M_n = L_n L_n^T$.
  
  \item Set $M_{n + 1} = L_n^T L_n$.
\end{enumeratenumeric}
Observe that all $M_n$ are similar to each other because $M_{n + 1} = L_n^{-
1} M_n L_n$. Each $M_n$ is also PSD. Therefore by the spectral theorem, the
stronger fact follows that all $M_n$ are orthogonally similar to each other.

Note that while there are more modern and sophisticated versions of the QR
algorithm {\cite{GoluVanl96}}, the original one is still widely taught, and
remains simpler than its more recent improvements. We think there are
conceptual advantages to better understanding the original one.

\section{Account of visualisation}

In this paper, we present a visualisation of the QR algorithm. Our
visualisation makes many properties of the naive QR algorithm, and its various
improvements, intuitive. Our visualisation may be beneficial to students and
researchers. Our visualisation is also qualitative, and using it we can show
that many features of the QR algorithm are in fact \tmtextit{universal} to
iterative eigenvalue algorithms, and not just specific to the QR algorithm.

The QR algorithm is an iterative eigenvalue algorithm. By iterative, we mean
that it employs a function $f$ from a set to itself, and evaluates it
repeatedly on some starting value $x$ to produce the sequence: $x, f (x), f (f
(x)), f (f (f (x))) \ldots$ until it converges close enough to a fixed point
of $f$. It's easy to verify that for the QR algorithm the fixed points are
matrices whose eigenvalues are easily found. The issue is that in general,
function iterations have very complicated dynamics. For some functions $f$ and
starting values $x$, the iteration can diverge, or even be chaotic. When
working over $\mathbb{R}$, the usual visualisation that's employed is
sometimes called the cobweb plot, and using it one can develop a good intution
for the dynamics of some instances of function iteration. We cannot use cobweb
plots to understand the dynamics of iterative \tmtextit{eigenvalue} algorithms
however, because even in the smallest non-trivial case, the $2 \times 2$ case,
the function $f$ maps $\mathbb{R}^4 \rightarrow \mathbb{R}^4$.

We focus only on PSD matrices (positive semi-definite). It can be argued that
this special case is sufficient for computing SVDs (Singular Value
Decompositions), eigendecompositions of \tmtextit{symmetric matrices}, and
eigendecompositions of orthogonal matrices.\footnote{Briefly: For symmetric
matrices, one can perform a shift $M' = M + \mu I$ for some easily chosen $\mu
\in \mathbb{R}$ which makes $M'$ PSD. For an orthogonal matrix $M$, we can use
the matrix logarithm (or Cayley Transform) as a first step towards making a
matrix $M'$ which is PSD. SVD is reducible to eigendecomposition of symmetric
matrices in various ways.} The PSD case is especially amenable to
visualisation. We won't try to justify further our focus on the PSD case.

The visualisation is based on the one-to-one correspondence between $n \times
n$ PSD matrices, and ellipsoids centred at the origin of $\mathbb{R^n}$. The
correspondence is given by $M \mapsto \{ M v \mid v \in S^n \}$ where $S^n$ is
the origin-centred hypersphere in $\mathbb{R^n}$. Every PSD matrix corresponds
to a unique ellipsoid, and every ellipsoid (that is origin-centred)
corresponds to a unique PSD matrices. We therefore sometimes talk about
ellipsoids and their semi-axes, instead of about matrices and their
eigen-values/vectors. We go as far as to say ``the ellipsoid'' instead of
``the PSD matrix'' sometimes.\footnote{ This all follows from the spectral
theorem.}

\begin{table}[h]
  \begin{tabular}{lll}
    \raisebox{0.0\height}{\includegraphics[width=3.14836678473042cm,height=2.94721566312475cm]{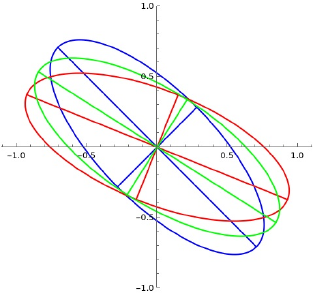}}
    &
    \raisebox{0.0\height}{\includegraphics[width=3.14836678473042cm,height=2.94721566312475cm]{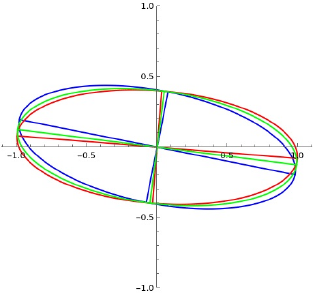}}
    &
    \raisebox{0.0\height}{\includegraphics[width=3.14836678473042cm,height=2.94721566312475cm]{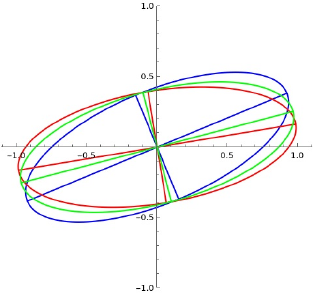}}\\
    \raisebox{0.0\height}{\includegraphics[width=3.14836678473042cm,height=2.94721566312475cm]{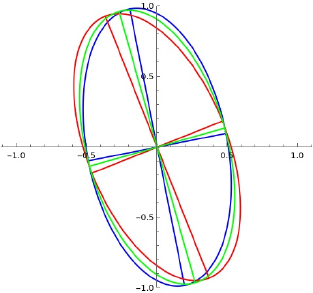}}
    &
    \raisebox{0.0\height}{\includegraphics[width=3.14836678473042cm,height=2.94721566312475cm]{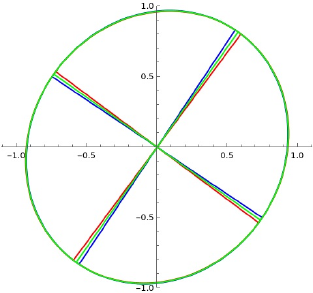}}
    &
    \raisebox{0.0\height}{\includegraphics[width=3.14836678473042cm,height=2.94721566312475cm]{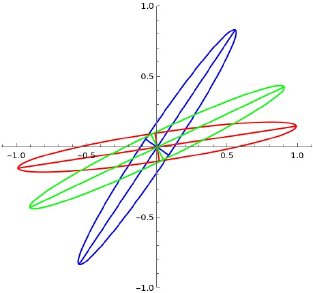}}
  \end{tabular}\tmcolor{blue}{\tmcolor{black}{{\begin{tabular}{l}
    \tmcolor{blue}{Input PSD matrix}\\
    \tmcolor{red}{1 iteration QR}\\
    {\color[HTML]{008000}1 iteration LR}
  \end{tabular}{}}}}
  
  \
  \caption{}
\end{table}

Our \tmcolor{blue}{input ellipse is in blue}. Our \tmcolor{red}{output ellipse
for 1 iteration is in red}. We also visualise
\tmcolor{green}{{\color[HTML]{008000}the LR algorithm in green} }for
comparison. We visualise the naive QR algorithm.

We make some immediate observations:
\begin{enumerate}
  \item The algorithm always converges for PSD matrices.
  
  \item In every iteration, the angle between the large semi-axes and the
  $x$-axis diminishes.
  
  \item The large semi-axis is in the same pair of quadrants for the output
  ellipse as for the input ellipse.
  
  \item The fixed points only occur when the semi-axes are aligned with the
  coordinate axes. For a given ellipse, this can happen two ways.
  
  \item If the large semi-axis is aligned with the $x$-axis, then this fixed
  point is stable.
  
  \item If the large semi-axis is aligned with the $y$-axis, then this fixed
  point is unstable. Being close to this fixed point is undesirable, because 1
  iteration moves you \tmtextit{away} from it, into the direction of the other
  fixed point. The closer you are to the unstable fixed point, the more slowly
  you move away from it.
  
  \item The rotation of the ellipse gets slower the closer it is to being a
  circle. This presents unsolvable difficulties for finding
  \tmtextit{eigenvectors} in general as opposed to \tmtextit{eigenvalues}.
  
  \item Opposite to point 7, the rotation of the ellipse gets faster the
  closer it is to being a degenerate line segment. This can be harnessed to
  speed up the algorithm.
  
  \item If the ellipse $M$ gets scaled to $M' = \lambda M$, then the angle of
  rotation remains the same.
  
  \item The input ellipse is congruent (in the sense of Euclidean geometry and
  not of linear algebra) to the output ellipse.
\end{enumerate}
Observation 7 has profound origins. It is a consequence of eigenvector
instability. When the geometric multiplicity of some eigenvalue $\lambda$ of a
matrix $M$ is 2 or more, then an infinitesimal perturbation of $M$ can
violently change the eigenspaces. Therefore, the slowness of rotation in the
near-circular case is a manifestation of an \tmtextit{unsolvable difficulty}
which is intrinsic to the eigendecomposition problem. Changing to a different
algorithm won't make it go away.

Further discussion of observation 7 requires us to know what a
\tmtextit{nearly diagonal matrix} looks like in the ellipse model:
\begin{enumerateroman}
  \item A \tmtextit{near-circle}.
  
  \item An ellipse whose semi-axes are nearly coordinate-axis aligned.
\end{enumerateroman}
These two scenarios (i and ii) have little overlap. We can find
\tmtextit{eigenvalues} in both scenarios (because of the Gershgorin circle
theorem), but \tmtextit{not eigenvectors}, for which we require scenario ii.
The eigenvector problem is uncomputable in general.

This uncomputability is perhaps not a ``big deal''. It often suffices to find
the eigendecomposition of a matrix close by to the input matrix, even though
this may \tmtextup{violently} change the eigenspaces. A discussion of why this
can ever be acceptable is left to a footnote.\footnote{ To understand why, we
recall the distinction between \tmtextit{forwards numerical stability} and
\tmtextit{backwards numerical stability}, and the difference in applications
of each. An algorithm $\widehat{f}$ for computing a function $f$ is
forwards-stable if $\widehat{f} (x) \approx f (x)$ for all $x$. An algorithm
$\widehat{f}$ for computing a function $f$ is backwards stable if given $x$
there is an $\widehat{x} \approx x$ such that $\widehat{f} (x) \approx f
(\widehat{x})$. Clearly, forwards stability is desirable, but we've shown that
for the eigendecomposition problem it is unattainable. Backwards stability
remains attainable, but what's the point of it? We explain. Let $f:  R
{\rightarrow} R$  be some continuous function. Eigendecomposition enables us to
extend $f$ to the set of PSD matrices by assuming that $f (P D P^{- 1}) = P f
(D) P^{- 1}$ is true. We can only compute eigendecompositions in a backwards
stable way, which means that we can make $M \approx P D P^{- 1}$ true but not
$M = P D P^{- 1}$. This is fine though, because by the continuity of $f$ we
have that $f (M) \approx P f (D) P^{- 1}$, so the approximation error from
computing $f$ via the eigendecomposition of $M$ can be made negligible. Be
careful though, because while $D$ may be close to the eigenvalue matrix of $M$
(because eigenspectra vary continuously) the same is \tmtextit{not} true for
$P$, which may be far from any true eigenvector matrix of $M$.}

The fix for observation 6 is to make the function being iterated
discontinuous. This is obvious in 2 dimensions thanks to the visualisation,
but it is also true in all dimensions. A more general version of this is the
statement of theorem \ref{must-be-discontinuous}. A continuous choice of
function $f$ might be nicer, but continuity can't hold
everywhere.\footnote{For the sake of conceptual understanding, the function
$f$ may be engineered to be continuous over a subspace of its domain over
which it has a globally attractive fixed point, while the discontinuities
outside of this subspace may serve to ``kick'' the input into the well-behaved
subspace. The Wilkinson shift happens to be discontinuous, so it provides the
needed discontinuities.}

Observation 8 appears to be a (kind of) converse to observation 7, but it's
difficult to verify under what conditions an iterative eigenvalue algorithm
behaves this way. In the case of the QR and LR algorithms, this behaviour may
be directly verified, but it would be interesting to demonstrate that it is
true for a large class of other algorithms.

\begin{table}[h]
  \begin{tabular}{lll}
    \raisebox{0.0\height}{\includegraphics[width=3.14836678473042cm,height=3.1920831693559cm]{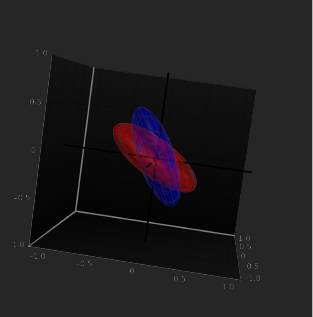}}
    &
    \raisebox{0.0\height}{\includegraphics[width=3.14836678473042cm,height=3.1920831693559cm]{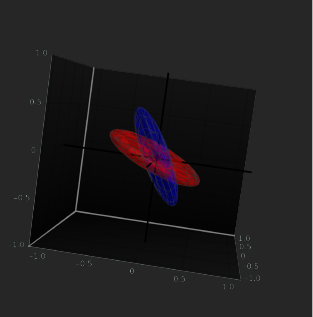}}
    &
    \raisebox{0.0\height}{\includegraphics[width=3.14836678473042cm,height=3.1920831693559cm]{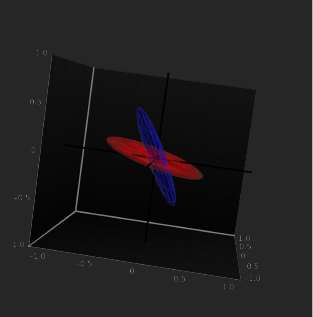}}
  \end{tabular}{\begin{tabular}{l}
    \tmcolor{blue}{Input PSD matrix}\\
    \tmcolor{red}{1 iteration QR}
  \end{tabular}{}}
  \caption{$3 \times 3$ PSD matrices. Observe that when pancaked, the
  ellipsoid aligns 1 semi-axis in 1 iteration.}
\end{table}

Observation 8 can be combined with observation 9 to motivate
\tmtextit{shifting}, and to demonstrate the limits of shifting and how they
may be bypassed through the use of \tmtextit{deflation}. Shifting is the map
$M' = M + \mu I$ for some arbitrary $\mu \in \mathbb{R}$. The effect this has
is to increase the lengths of all the semi-axes by $\mu$. To better understand
this effect, we exploit observation 9 to scale down the ellipse so that its
largest semi-axis has length exactly equal to $1$. We then observe that the
effect of shifting is that it either fattens or thins the ellipse. The fattest
case is a perfect circle, which by observation 7 is the worst-case scenario.
The thinnest case is a line segment, for which convergence is instant. In
higher dimensions, the extreme cases are \tmtextit{flat pancakes} or perfect
hyperspheres. Therefore, shifting is pancaking! Pancakes are, of course, the
best; and blobs are, of course, the worst. But once you've got a perfect
pancake, shifting runs out of steam as a speeding-up technique, so we may
employ \tmtextit{deflation} to destroy the dimension of the ellipsoid that has
been reduced to $0$, so that we may pancake the remaining dimensions, and
rekindle the speed-up.

The above analysis of observation 8 shows that iterative eigenvalue algorithms
in which observation 8 holds can be used as improvers of cruder but faster
eigenvalue estimation technique. If an eigenvalue estimation technique can
crudely estimate the \tmtextit{smallest} eigenvalue, then this can be used as
the value of the shift $\mu$, which can be used then to pancake the ellipsoid
effectively, and which in turn will accelerate the iterative eigenvalue
algorithm, and which may \tmtextit{in turn} make the crude estimation
technique more accurate in subsequent iterations, and so on.

\section{Account of generalisation to $n$ dimensions}

It may not be obvious that the conclusions generalise straightforwardly to PSD
matrices in $n$ dimensions, but they do:
\begin{itemizedot}
  \item The algorithm converges for all PSD matrices.
  
  \item The stable fixed points are precisely the diagonal matrices of the
  form $\tmop{diag} (\lambda_n, \lambda_{n - 1}, \ldots, \lambda_1)$ where all
  the $\lambda_i$'s are non-negative reals, and $\lambda_n \geq \lambda_{n -
  1} \geq \ldots \geq \lambda_1$. The other diagonal matrices are unstable
  fixed points, with convergence slowing down the smaller the angle an
  ellipsoid makes with them.
  
  \item The output of each iteration is orthogonally similar (as a matrix) to
  its input. In terms of Euclidean geometry, it is a congruent ellipsoid. (Be
  aware that ``congruence'' in Euclidean geometry means something different
  from in linear algebra).
  
  \item At eigenvalue clashes, the eigenvectors are unstable, but the
  algorithm may still converge quickly to a nearly diagonal matrix. In which
  case, only the eigenvalues and \tmtextit{not the eigenvectors} may be
  obtained under such circumstances.
  
  \item Shifting can be understood as ``pancaking'' the ellipsoid (once the
  semi-axes are scaled so that the largest semi-axis has length $1$). Like in
  the 2D case, this speeds up convergence.
  
  \item Deflation can be thought of as continuing the ``pancaking'' (shifting)
  once one of the semi-axes has been fully pancaked.
\end{itemizedot}
These conclusions can be checked by doing the ellipsoid visualisation in 3
dimensions. But this isn't necessary, as the 2D case is suggestive enough.

\section{Remark about the dual purpose of Wilkinson shifts, and the search for
better shifting strategies}

Wilkinson shifts are curiously enough not continuous. Wilkinson shifts appear
to both ``kick'' the ellipsoid away from the unstable fixed point (if it's too
close to it) and to ``pancake'' the ellipsoid. Both of these features speed up
convergence. One might wonder though if these two properties: The kicking away
from the unstable fixed point, and the pancaking of the ellipsoid, can be
achieved separately. In particular, the shift $\mu$ can perhaps be made into a
continuous function, with it serving only to ``pancake'' the ellipsoid. This
might simplify the search for superior shifting strategies if we can narrow
the search down to only continuous functions, with the discontinuity needed to
do the ``kicking'' from the fixed point coming from elsewhere.

\section{Axiomatic development of the theory of iterative eigenvalue
algorithms}

One advantage of the visualisation is that it is mostly qualitative. In
particular, we see that the LR algorithm exhibits the same qualitative
behaviour as the QR algorithm. We exploit this to show that much of the
behaviour and theory of the QR algorithm is not caused by the QR decomposition
as such, but is intrinsic to solving the eigenvalue estimation problem by the
use of fixed-point iteration.

Consider the following axioms for a fixed-point iteration algorithm:
\begin{enumeratenumeric}
  \item It consists of iterating some function $f$ over PSD matrices.
  
  \item $f (M)$ is orthogonally similar to $M$.
  
  \item The sequence $(f^n (M))_{n \in \mathbb{N}}$ converges for every $M$.
  It furthermore converges to a fixed point of $f$.
  
  \item A fixed point can only be a diagonal matrix.
  
  \item All fixed points are attractive. (This makes convergence fast).
\end{enumeratenumeric}
We consider these 5 axioms to be highly desirable. Intriguingly, these rule
out the possibility that $f$ can be continuous everywhere. This is the
statement of theorem \ref{must-be-discontinuous}. Morally, the theorem holds
because given a non-scalar matrix $M$, the topological space of matrices
orthogonally similar to $M$ is connected, but has some ``holes'' in it. The
influence these ``holes'' have is that they prevent axioms 1 to 5 being
realised when $f$ is continuous, even though these 5 conditions are highly
desirable. A work-around might be to make $f$ continuous over some
neighbourhood of each fixed point, and use the discontinuity of $f$ outside
these neighbours to make $f (x)$ only map into these neighbourhoods.

\begin{theorem}
  \label{must-be-discontinuous}Assuming axioms 1 to 5, the function $f$ being
  iterated must have discontinuities in the set of PSD matrices.
\end{theorem}

\begin{proof}
  We restrict the domain and codomain of $f$ to the set of matrices
  orthogonally similar to some PSD matrix $M$, where $M$ is not a multiple of
  the identity matrix. The fact that we can do this follows from condition 2.
  {\tmstrong{We show that $f$ has at least two fixed points under
  continuity}}: Assume it only has one fixed point $x$. It must be attractive
  by condition 5. By condition 3, all points attract to $x$. This results in
  the space of matrices orthogonally similar to $M$ being a contractible
  space, which it isn't. We get a contradiction. Therefore $f$ has at least
  two fixed points orthogonally similar to $M$.
  
  {\tmstrong{Now we show that under continuity, one of these two fixed points
  is not attractive, which contradicts conditions 3 and 5}}. The space of
  matrices orthogonally similar to some arbitrary matrix $M$ is connected. The
  set of points which attract to some attractive fixed point is always open.
  The basins of attraction of the attractive fixed points are disjoint, so by
  topological connectivity there are points not in their union, which are the
  points which don't attract to an attractive fixed point. This contradicts
  conditions 3 and 5.
\end{proof}

\section{Future work regarding visualisations}

We ask the following questions. Forgive us if some happen to be easy:
\begin{itemize}
  \item Can a visualisation technique for iterative eigenvalue algorithms be
  developed for non-symmetric matrices? The dynamics of the QR algorithm in
  the case of non-symmetric matrices (especially featuring shift policies) is
  not as well understood as in the symmetric case {\cite{banks2021global}}.
  
  \item Can a visualisation technique help to better understand Hessenberg
  matrices and their uses for finding eigenvalues and eigenvectors? We haven't
  succeeded yet in producing anything useful for this. Even symmetric
  Hessenberg matrices would be interesting.
\end{itemize}
It's important that visualisations be legible. Whether a visualisation is
legible or not is perhaps subjective.

\section{Novelty}

We first put a visualisation of the QR algorithm onto the Wikipedia page on
the QR algorithm on August 2021. We did so under a pseudonym. Afterwards, we
left it for some time. We then later presented some of the material in this
paper in a Youtube video series {\cite{video}}, also under a pseudonym. Based
on a private correspondence with Professor David Watkins {\cite{david-email}}
(who has written numerous pedagogical papers on iterative eigenvalue
algorithms like the QR algorithm {\cite{watkins-tutorial2,watkins-tutorial1}})
the presentation is likely a novel one. The emphasis is likely novel. Videos
and lectures are an important accompaniment when trying to present material
like the one here.

\end{document}